\documentclass[amscd,amssymb,verbatim,12pt]{amsart}
\input epsf
\usepackage{newlfont}
\usepackage{subfigure}
\usepackage{graphicx}
\usepackage{caption}
\usepackage{color}
\usepackage[percent]{overpic} 

\usepackage{mathrsfs}

\newcommand{\pubnote}[1]{}  
\newcommand{\hide}[1]{}   

\newcommand{\F}[1]{Figure~\ref{F:#1}}         

\newcommand{\bbb}{\mathbb}      
\newcommand{\CP}{\text{\tt CirclePack}} 

\newcommand{\etl}[1]{#1}        
\newcommand{\ctl}[1]{\mathpzc{#1}}        
\newcommand{\ktl}[1]{\mathtt{#1}}        
\newcommand{\ptl}[1]{\mathtt{#1}}        
\newcommand{\atl}[1]{\mathpzc{#1}}        
\newcommand{\wmD}{\widetilde{D}}  %
\newcommand{\wmC}{\widetilde{\mathcal{C}}}  %

\newcommand{\cis}{\xrightarrow{shape}} 
\newcommand{\bs}[1]{^{\langle#1\rangle}}  
\newcommand{\bup}[2]{{#1}^{#2}}     

\newcommand{\cD}{D}

\newcommand{\bC}{\bbb C}         
\newcommand{\bD}{\bbb D}         

\newcommand{\bG}{\bbb G}         

\newcommand{\mC}{\mathcal C}

\DeclareFontFamily{OT1}{pzc}{}
\DeclareFontShape{OT1}{pzc}{m}{it}{<-> s * [1.10] pzcmi7t}{}
\DeclareMathAlphabet{\mathpzc}{OT1}{pzc}{m}{it}

\theoremstyle{plain}
\newtheorem{Thm}{Theorem}[section]
\newtheorem{Lem}[Thm]{Lemma}
\newtheorem{Claim}[Thm]{Claim}

\newtheorem{Question}{Question}

\newtheorem*{definition}{Definition}

\theoremstyle{definition}

\begin{document}


\title{Shape Convergence for Aggregate Tiles\\in Conformal Tilings}
\author{R. Kenyon}
	\address{Brown University, Providence, RI}
	\email{rkenyon at math.brown.edu}
\author{K. Stephenson}
     \address{University of Tennessee, Knoxville} 
     \email{kens at math.utk.edu}



\begin{abstract}
Given a substitution tiling $\etl{T}$ of the plane with
subdivision operator $\tau$, we study the conformal tilings $\ctl{T}_n$
associated with $\tau^n\etl{T}$. We prove 
that aggregate tiles within $\ctl{T}_n$
converge in shape as $n\rightarrow \infty$ to their associated
Euclidean tiles in $\etl{T}$.
\end{abstract}

\maketitle

\section{Introduction}

The term {\sl tiling} refers to a locally finite decomposition of a
topological plane into a pattern of compact regions known as its {\sl
  tiles}. This paper involves tilings of four successive types: Starting from a {\sl substitution} tiling, one can construct a {\sl
  combinatorial} tiling, then an {\sl affine} tiling, and finally a
{\sl conformal} tiling. Connections between the first and the last
are the subject of this paper, with the middle two as necessary
bridges. Here are the rough definitions of these objects.

\vspace{10pt}
{\narrower{\setlength{\parindent}{10pt}
           
\noindent $\bullet$ {\sl Substitution tilings $\etl{T}$:} This is a
well-studied class of Euclidean tilings of the complex plane. Each
tile in a substitution tiling $\etl{T}$ is similar to one of a finite
set of polygonal prototiles. Moreover, there is an associated
subdivision rule $\tau$ specifying how each tile can be decomposed
into a finite union of subtiles, each again similar to one of the
prototiles. The tilings themselves are limits of successive
subdivisions followed by rescaling.

\vspace{6pt} 
\noindent $\bullet$ {\sl Combinatorial tilings $\ktl{K}$:} These are
abstract cell decompositions of a topological plane obtained by
removing the metric properties from substitution tilings $\etl{T}$.
The geometric subdivision rule $\tau$ for $\etl{T}$ becomes a
combinatorial subdivision rule which can be applied to $\ktl{K}$.

\vspace{6pt}
\noindent $\bullet$ {\sl Affine tilings $\atl{A}$:} These are obtained
from combinatorial tilings $\ktl{K}$ by identifying each $n$-sided
cell of $\ktl{K}$ with a unit-sided regular Euclidean $n$-gon.  An
affine tiling $\atl{A}$ is not realized metrically in $\bC$, but
defines rather a plane with a piecewise Euclidean metric structure.

\vspace{6pt} 
\noindent $\bullet$ {\sl Conformal tilings $\ctl{T}$:} These arise
from affine tilings by imposing a canonical conformal structure in
which each tile is conformally regular and enjoys a certain
anticonformal {\sl reflective property} across its edges.  The
resulting Riemann surface is conformally equivalent to
$\bC$, and its image under a conformal homeomorphism is what
we refer to as a {\sl conformal tiling} $\ctl{T}$. 

}}
\vspace{10pt}

We use distinct symbols to distinguish these four categories
and the symbol ``$\sim$'' to denote corresponding objects. A substitution
tiling $\etl{T}$ leads to a combinatorial tiling $\ktl{K}$, then to an
affine tiling $\atl{A}$, and finally to a conformal tiling $\ctl{T}$.
Thus $\etl{T}\sim\ktl{K}\sim\atl{A}\sim\ctl{T}$. A tile
$\etl{t}\in\etl{T}$ is a Euclidean polygon, and corresponds to a tile
$\ktl{k}\in\ktl{K}$, a combinatorial $n$-gon, which in turn
corresponds to a tile $\atl{a}\in\atl{A}$, a regular Euclidean
$n$-gon, and this finally corresponds to a conformal tile
$\ctl{t}\in\ctl{T}$, which is a {\sl conformal polygon}, that is, a
topological polygon in the plane with analytic arcs as sides. Thus
$\etl{t}\sim\ktl{k}\sim\atl{a}\sim\ctl{t}$. We will shortly review the
definitions and properties of tilings, and in particular of conformal
tilings as developed in \cite{BS17}.

A substitution tiling $\etl{T}$ comes with a subdivision operator
$\tau$, and applying $\tau$ leads to a new substitution tiling
$\tau\etl{T}$ with the same set of prototiles. That subdivision operation is
also inherited by the associated combinatorial tiling $\ktl{K}\sim\etl{T}$,
and we have $\tau\ktl{K}\sim\tau\etl{T}$.
In both these settings, $\tau$ is considered an {\bf in situ} operator, that is,
it subdivides tiles in place within $\etl{T}$ or $\ktl{K}$. Figure \ref{F:Pinwheel} 
illustrates a fragment of the pinwheel tiling and its subdivision rule.
\vspace{10pt}
\begin{center}
\begin{overpic}[width=.75\textwidth
]{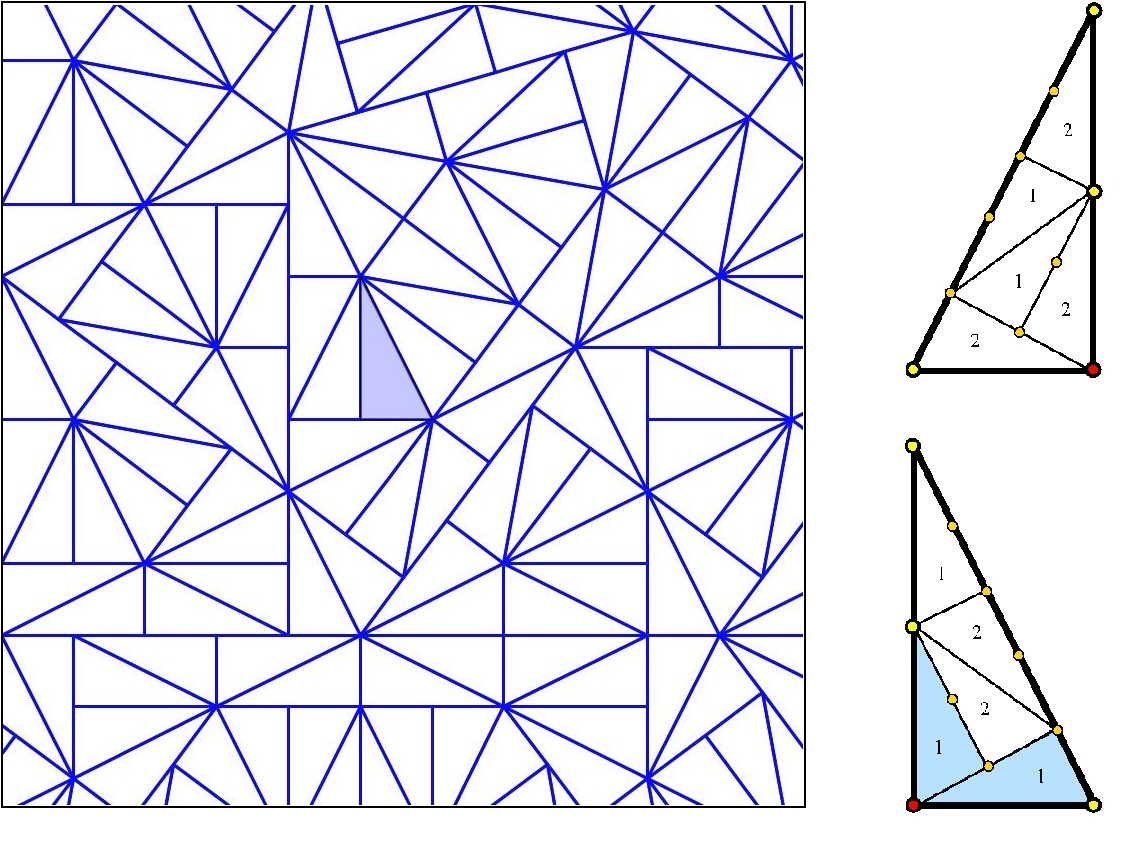}
\put (82,63) {$\ptl{p}_1$}
\put (92,25) {$\ptl{p}_2$}
\end{overpic}
\captionof{figure}{A fragment of a ``pinwheel'' substitution tiling
  of the plane, its two prototiles, and the subdivision rule $\tau$.}
\label{F:Pinwheel}
\end{center}

\vspace{10pt} Things are different in the conformal setting. True,
there is a conformal tiling $\ctl{T}\sim\ktl{K}$, and a conformal
tiling $\ctl{T}_1\sim \tau\ktl{K}$. It is not appropriate, however, to
write $\ctl{T}_1=\tau\ctl{T}$, since $\tau$ does not generally act
{\sl in situ} in the conformal case. A conformal tile $\ctl{t}$ is
generally different in shape from its Euclidean counterpart
$\etl{t}\sim \ctl{t}$, and the unions of conformal tiles coming from
successive subdivisions of $\ctl{t}$ will have a succession of yet
other shapes. However, experiments in \cite{BS17} (see \S3.6) suggested
that aggregate tiles --- the union of tiles associated with successive
subdivisions of a given tile --- look increasing like that tile's
Euclidean counterpart. In this paper we prove that this is indeed
the case. Note in particular that the purely combinatorial tiling
$\ktl{K}\sim\etl{T}$ and the combinatorial subdivision rule $\tau$
somehow encode all the geometric information in $\etl{T}$ itself.
This and other comments about our result will be discussed in
Section~\ref{S:Examples}.

\section{Example: the Pinwheel Tiling}
An early example may be helpful. The pinwheel tiling $\etl{T}$ was
introduced by John H. Conway (see Charles Radin \cite{cR94}).  The tiles are
all [1:2:$\sqrt{5}$] triangles. Due to orientation, there are two
prototiles, $\{\ptl{p}_1,\ptl{p}_2\}$, pictured along with the
associated subdivision rule $\tau$ on the right in \F{Pinwheel}. Note
that $\tau$ breaks each triangle into 5 similar triangles, their types
indicated by ``1'' and ``2''. The two shaded subtiles in $\ptl{p}_2$
will be discussed shortly. \F{PinPairs} focuses on root tile $\etl{t}$,
the shaded tile in \F{Pinwheel}.

\begin{center}
\begin{overpic}[width=.8\textwidth
]{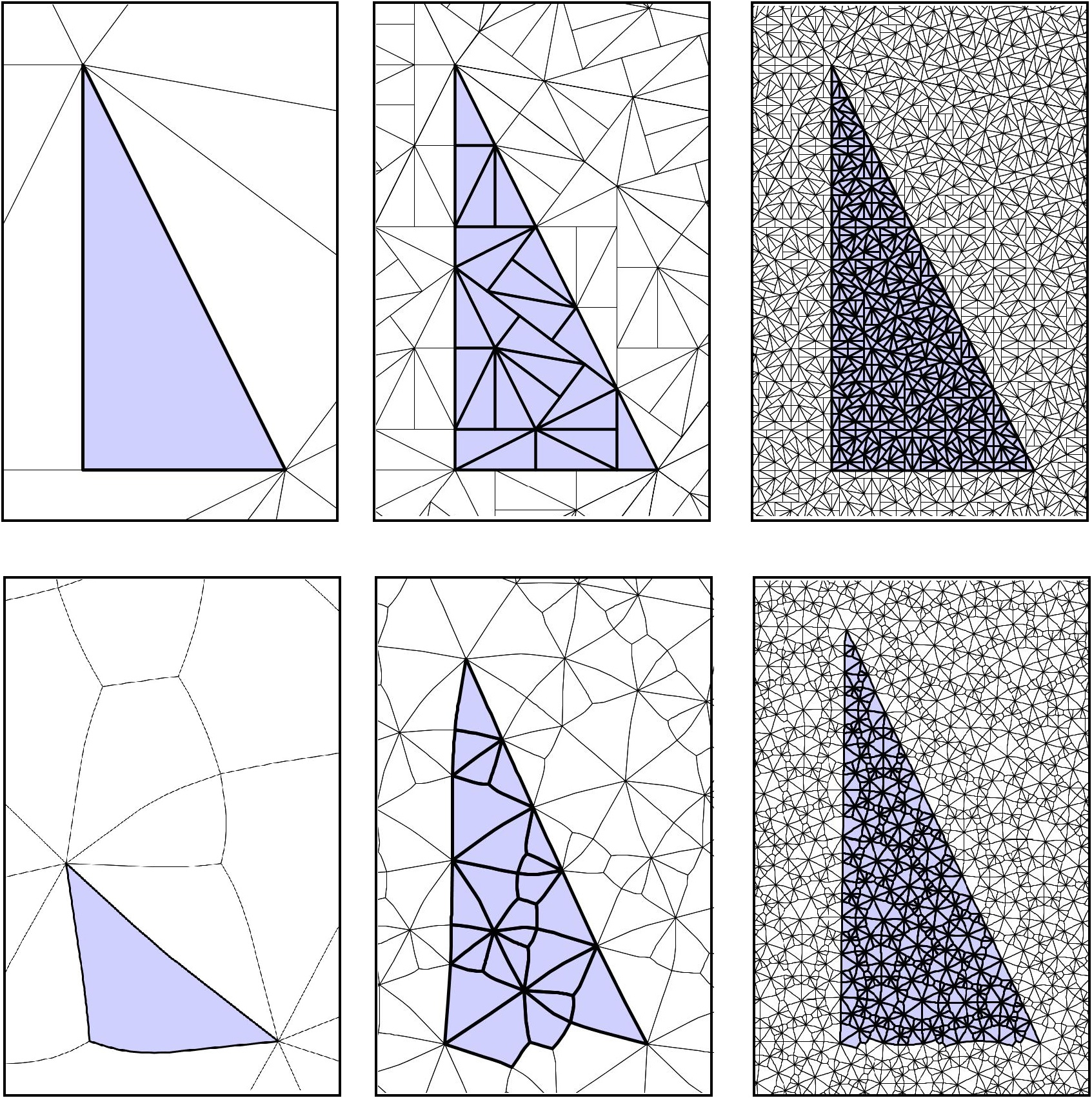}
\put (6,54.3) {$a$}
\put (26,54.3) {$b$}
\put (6,2.5) {$a$}
\put (26,2.5) {$b$}
\end{overpic}
\captionof{figure}{Across the top: Euclidean tile $\etl{t}$ and
  aggregates $\etl{t}\bs{2}$, $\etl{t}\bs{4}$. Across the bottom: conformal
  versions $\ctl{t}$, $\ctl{t}\bs{2}$, $\ctl{t}\bs{4}$.} \label{F:PinPairs}
  \hide{op are a Euclidean pinwheel tile
  $\etl{t}$, its 2-aggregate tile $\etl{t}\bs{2}$ after 2 subdivsions,
  and its 4-aggregate tile $\etl{t}\bs4$ after 4 subdivisions.  Along
  the bottom are the associated conformal objects, $\ctl{t}$,
  $\ctl{t}\bs2$, and $\ctl{t}\bs4$.}
\end{center}

The top row of \F{PinPairs} shows $\etl{T}$, the twice subdivided
tiling $\tau^2\etl{T}$, and the four times subdivided tiling
$\tau^4\etl{T}$.  The subdivision operation occurs {\it in situ} in
the Euclidean setting, so $\etl{t}$ is a union of its 25 subtiles in
$\tau^2\etl{T}$ and a union of its $625$ subtiles in $\tau^4\etl{T}$.

The bottom row of \F{PinPairs} shows the corresponding fragments of
the associated conformal tilings, denoted $\ctl{T}, \ctl{T}_2,$ and
$\ctl{T}_4$, respectively.  Each conformal version is normalized so
the base corners are at 0 and 1, as in the top row. The tiles
associated with $\etl{t}$ are again highlighted in blue.  Note that
subdivsion in the conformal setting does not happen {\it in situ}:
each subdivision of $\etl{T}$ engenders a new conformal structure. The
25 blue tiles in $\ctl{T}_2$ form what we call a $2$-aggregate tile
$\ctl{t}\bs2$, while the 625 blue tiles from $\ctl{T}_4$ form a
$4$-aggregate tile $\ctl{t}\bs4$.

The point of our paper is hinted at in the two tilings in the right
column of \F{PinPairs}: although the subdivision action is not {\it in
  situ} in the bottom, the conformal $n$-aggregate tiles ---
aggregated after $n$ stages of subdivision --- seem to be converging
in shape to the original Euclidean tile $\etl{t}$. After some notation
and definitions, we state a theorem which confirms this convergence.

\section{Tiling Details}
The substitution tilings $\etl{T}$ of interest here are {\sl
  aperiodic}, {\sl hierarchical} tilings of $\bC$ displaying {\sl
  finite local complexity}.  The most well known examples are the
Penrose tilings, which will appear along with ``chair'', ``domino'',
and ``sphinx'' examples in \S{\ref{S:Examples}}.  The reader may refer
to \cite{BS17} for further background. Briefly, each substitution tiling
$\etl{T}$ has an associated finite set
$\{\ptl{p}_1,\ptl{p}_2,\cdots,\ptl{p}_q\}$ of Euclidean polygonal
{\sl prototiles}, with every tile $\etl{t}\in\etl{T}$ being similar to one
of the prototiles. We note that prototiles are distinguished by shape,
orientation, a designated base edge $\langle a,b\rangle$, and possibly
by some abstract ``label''. By a ``similarity'' of tiles, we mean a
geometric similarity as polygons that also respects these features.
Every prototile has an associated decomposition into subtiles, each
subtile again being similar to one of the prototiles. These
decompositions define a {\sl subdivision} operator $\tau$ which, when
applied to $\etl{T}$ yields a new substitution tiling $\tau\etl{T}$
with the same prototiles. In the traditional theory, subdivision may
be accompanied by an associated renormalization that is, (rescaling);
that will not be the case here, so $\tau$ operates as an {\it in situ}
subdivision operator. 

As noted earlier, if we ignore the metric properties of $\etl{T}$, we
are left with a cell decomposition of a topological plane, which we
treat as a combinatorial tiling and denote by $\ktl{K}$. To avoid
ambiguity, we require a condition of $\ktl{K}$: any two of
its tiles are either disjoint or their intersection is a union of
vertices and/or full edges. Note in the pinwheel subdivision rule of
\F{Pinwheel}, for example, that although the tiles are Euclidean
triangles, the pattern of tile intersections requires four designated
corners (marked by dots), and pinwheel tiles are thus considered 
$4$-sided.
The faces of $\ktl{K}$, which are combinatorial polygons, retain their associations
with prototiles so that $\tau$ may, abusing notation, be treated as a
combinatorial subdivision operator: thus $\ktl{K}\sim\etl{T}$ implies
$\tau\ktl{K}\sim\tau\etl{T}$.

With $\ktl{K}$ in hand, we consider the associated affine tiling
$\atl{A}$. Every tile $\ktl{k}\in \ktl{K}$ may be identified with a
unit sided regular equilateral $n$-gon $p$, where $n$ is the number of
sides of $\ktl{k}$. If tiles $\ktl{k}_1,\ktl{k}_2$ share an edge, then
their polygons $p_1,p_2$ are identified isometrically along the
corresponding edge. This allows us to define a piecewise Euclidean
metric on $\ktl{K}$ with conical singularities. Thus we arrive at our
affine tiling $\atl{A}$, writing $\atl{A}\sim\ktl{K}$. Note that the
corners of tiles of $\atl{A}$ are generically non-flat cone points, so
we do not attempt to realize $\atl{A}$ in any concrete setting.

This brings us to the tiling of direct interest, the conformal tiling
$\ctl{T}$. There is a canonical conformal atlas on $\atl{A}$:
interiors of the regular polygons provide charts for their points,
interiors of unions of two adjacent regular polygons provide charts
for interior points of tile edges, and local power maps provide charts
for tile vertices. With this atlas, the tiling $\atl{A}$ becomes a
Riemann surface. Since $\atl{A}$ is simply connected and not compact,
the classical Uniformization Theorem in complex analysis implies
existence of a conformal homeomorphism $\phi:\atl{A}\rightarrow \bG$
where $\bG$ is either the unit disc $\bD$ or the complex plane $\bC$.

\begin{definition} Given an affine tiling $\atl{A}$, the images
  $\{\phi(a):a\in\atl{A}\}$ under a conformal homeomorphism
  $\phi:\atl{A}\rightarrow \bG$ form a {\bf conformal tiling} in
  $\bG$. We write $\ctl{T}$ for this tiling and note that it
  is uniquely determined up to M\"obius transformations of $\bG$.
\end{definition}

In a sense, $\atl{A}$ with its conformal structure is already a
conformal tiling. However, under the map $\phi$, the individual tiles
become concrete shapes in $\bG$, and it is these shapes which are of
interest in our work. The properties of $\ctl{T}$ and its tiles are
developed fully in \cite{BS17}. We need not be concerned with details, but
some features are noteworthy: Each conformal tile $\ctl{t}\in \ctl{T}$
is a {\sl conformal polygon}, a curvilinear polygon with sides which
are analytic arcs. Indeed, each is a {\sl conformally regular
  polygon}, meaning that there is a conformal self-map
$f:\ctl{t}\rightarrow \ctl{t}$ which maps each corner to the next; $f$
has a single fixed point, the {\sl conformal center} of
$\ctl{t}$. Conformal tiles $\ctl{t},\ctl{t}'$ sharing an edge have a
anti-conformal reflective relationship across that edge, which leads
to an important rigidity phenomenon in conformal tilings: the shape of
any single tile of $\ctl{T}$ determines uniquely the shapes and
locations of every other tile of $\ctl{T}$.

There remains the issue of whether $\ctl{T}$ lies in $\bD$ or
$\bC$. This is known as the ``type'' problem and might normally
require some work to resolve. For substitution tilings of the plane
with finite local complexity, however, the associate conformal tilings
are always {\sl parabolic}, that is, $\bG=\bC$. This will follow
from the quasiconformal arguments in Lemma~\ref{L:qc} below.

\section{Subdivision Details}\label{S:subdetails}
Our substitution tilings $\etl{T}$ are assumed to display {\sl finite
  local complexity}. This means simply that any two tiles can be
juxtaposed in at most finitely many ways, up to similarity. The same
then holds for any subdivision $\tau^n\etl{T}$.
We must place a side condition on $\etl{T}$, one that is nearly universal,
failing in only the most trivial of substitution tilings. Nevertheless,
we make it explicit for later use.

\vspace{10pt}
\noindent{\bf Standing Assumption on $\etl{T}$:} {\sl There must exist
  a configuration $\mC=\etl{p}\cup\etl{q}$, a union of two Euclidean tiles, 
so that the following holds:
\begin{enumerate}
\item{both $\etl{p}$ and $\etl{q}$ are congruent to the same
  prototile $\ptl{p}$;}
\item{if $S:\etl{p}\rightarrow \etl{q}$ is a congruence, then the linear
part of $S$ is not $\pm I$, plus or minus the identity.}
\item{for in every open disc $D=D(r,z)\subset \bC$ there exists an
  integer $n$ so that $\tau^n\etl{T}$ contains a pair of tiles,
  $\etl{t}_1,\etl{t}_2$ having the type of $\ptl{p}$ so that their
  union $\etl{t}_1\cup\etl{t}_2$ is similar to $\mC$.}
\end{enumerate}
}

\vspace{10pt} In the tile schematics of \F{Pinwheel}, the union of the
two shaded tiles is an example of such a configuration $\mC$: both
tiles are congruent to $\ptl{p}_1$ and are not translations of one
another. A similar configuration will occur in the subdivision of
every tile similar to $\ptl{p}_2$ and thus will occur densely
throughout the plane as $\etl{T}$ undergoes subdivision. In general,
of course, the tiles forming $\mC$ are not necessarily contiguous or
subtiles of the same parent.

\section{Statement of the Theorem}
A substitution tiling $\etl{T}$ is completely compatible with its
subdivision operator $\tau$ in that $\tau \etl{T}$ is an {\sl in situ}
decomposition of $\etl{T}$ into subtiles.  More generally, for every
positive integer $n$, $\tau^n \etl{T}$ is an {\sl in situ}
decomposition of $\tau^{n-1}\etl{T}$, and hence by induction, of
$\etl{T}$ itself. For convenience we write $\etl{T}_n$ for
$\tau^n\etl{T}$. And if $\ktl{K}\sim \etl{T}$, then we write
$\ktl{K}_n=\tau^n\ktl{K}$, noting that $\ktl{K}_n\sim\etl{T}_n$.

The reverse of subdivision is {\sl aggregation}. Fix $n\ge 0$ and
focus on a combinatorial tile $\ktl{k}\in \ktl{K}_n$. Now perform an
additional $m$ subdivisions of $\ktl{K}_n$ to get $\ktl{K}_{n+m}$.
Write $\ktl{k}\bs{m}$ for the union of tiles in $\ktl{K}_{n+m}$ which
were generated during the $m$ subdivisions of $\ktl{k}$: this union
$\ktl{k}\bs{m}$ will be called an {\sl $m$-aggregate tile} in
$\ktl{K}_{n+m}$ and the combinatorial connection to $\ktl{k}$ is
indicated by writing $\ktl{k}\bs{m}\sim \ktl{k}$. In other words,
an $m$-aggregate tile is a union of tiles at the $(n+m)$th subdivision
stage which form a single tile $\ktl{k}$ from the $n$th subdivision
stage.

This notion of aggregation (and the associated notations) apply
equally to substitution, combinatorial, affine, and conformal
tilings. However, the geometric differences are the subject of this
paper. Consider a substitution tiling $\etl{T}$, a tile $\etl{t}\in
\etl{T}_n$, the combinatorial tile $\ktl{k}\sim
\etl{t}$ in $\ktl{K}_n$, and its $m$-aggregate
$\ktl{k}\bs{m}$. By definition, $\ktl{k}\bs{m}=\ktl{k}$. Likewise,
because subdivision occurs {\sl in situ} for substitution tilings, the
$m$-aggregate tile $\etl{t}\bs{m}\sim \ktl{k}\bs{m}$ is equal as a point
set to the original tile $\etl{t}\in \etl{T}_n$.

This is not the case for the corresponding conformal tiles. Let
$\ctl{T}$, $\ctl{T}_n$, and $\ctl{T}_{n+m}$ be conformal tilings,
where $\ctl{T}\sim\ktl{K}$, $\ctl{T}_n\sim\ktl{K}_n$, and
$\ctl{T}_{n+m}\sim\ktl{K}_{n+m}$, $n,m\ge 0$.  Consider
$\ctl{t}\in\ctl{T}_n$, with $\ctl{t}\sim \ktl{k}$. In general, the
union of conformal tiles of $\ctl{T}_{n+m}$ corresponding to the
$m$-aggregate tile $\ktl{k}\bs{m}$ is {\sl not equal} as a point set
to the conformal tile $\ctl{t}\in \ctl{T}_n$. This can be seen in the
conformal images of the bottom row in \F{PinPairs}: the 2-aggregate
$\ctl{t}\bs{2}$ shown in the middle image, a union of 25 blue
conformal tiles, has a shape unequal to its parent conformal tile,
shown in the image to its left. Likewise, the 4-aggregate
$\ctl{t}\bs{4}$ shown on the right has yet another shape. Comparing
the right side images from the top and bottom rows, however, the
shapes appear to be getting close. Before our statement, we need to
formalize a notion of shape.

\begin{definition} Jordan domains $\Omega_j$ in the plane are said to
  {\bf converge in shape} to a Jordan domain $\Omega$ if there exist
  Euclidean similarities $\Lambda_j$ of the plane so that
  $\Lambda_j(\Omega_j)$ converges to $\Omega$ in the Hausdorff metric
  as $j\rightarrow \infty$. Write $\Omega_j\cis \Omega$.
\end{definition}

Conformal tilings can be approximated in practice using methods of
circle packing. Experiments carried out by the third author and Phil
Bowers in \cite{BS17} suggested that the phenomenon of
shape evolution seen with the pinwheel tiling is not unique. A peek
ahead to \F{AggShapes} reveals shape comparisons for some other well
known substitution tilings. The experiments that gave us these
images led to the shape question (\cite[page 38]{BS17}) which is answered
affirmatively here.

\begin{Thm} Let $\etl{T}$ be a substitution tiling of the plane with
  subdivision rule $\tau$ and let $\ctl{T}_n$ denote the conformal
  tilings associated with $\tau^n\etl{T}$. Fix a tile
  $\etl{t}\in\etl{T}$. For each $n\ge0$ let $\ctl{t}\bs{n}$ be the
  associated $n$-aggregate tile in $\ctl{T}_n$. Then
  $\ctl{t}\bs{n}\cis \etl{t}$ as $n\rightarrow \infty$.
\end{Thm}

\section{Proof} 
We will work with various homeomorphisms between tilings. These
will be termed {\sl tiling maps}, in that they carry each tile
of the domain tiling bijectively to the associated tile of the range
tiling. Each tile is identified with some prototile, and 
tiling maps are always assumed to respect tile types and 
designated base edges $\langle a, b\rangle$.

We are free to normalize our tilings with similarities, so
assume that $\etl{T}$ is positioned so that the tile $\etl{t}$ of
interest has two of its corners $a,b$ at $0,1$, respectively. Subdivision
of $\etl{T}$ occurs in place, so these corners of the aggregate tiles
$\etl{t}\bs{n} \subset \etl{T}_n$ remain at 0 and 1. On the conformal
side, for each $n$ we apply a similarity to put the corresponding
corners of the aggregate conformal tile $\ctl{t}\bs{n}\sim\etl{t}\bs{n}$
at 0 and 1 as well. Quasiconformal maps $F_n:\etl{T}_n\rightarrow \ctl{T}_n$
are key to our proof.

\begin{Lem}\label{L:qc}
  Let $\etl{T}$ be a substitution tiling and $\ctl{T}$
  the associated conformal tiling. Then there exists a
  $\kappa$-quasiconformal tiling map $f:\etl{T}\rightarrow \ctl{T}$,
  where $\kappa=\kappa(\etl{T})$ depends only on the finite set of
  prototiles for $\etl{T}$. In particular, $\ctl{T}$ is parabolic.
\end{Lem}

\begin{proof}
We define $f$ {\sl via} an intermediate tiling map $g:\etl{T}
\rightarrow \atl{A}$, which is constructed tile-by-tile based on
prototile type.

Consider a prototile $\ptl{p}$ having $m$ vertices. Define in {\sl ad
  hoc} fashion a straight-edge triangulation of $\ptl{p}$ by adding a
finite number of vertices and Euclidean line segments as necessary in
the interior of $\ptl{p}$.  Any tile $\atl{a}\in\atl{A}$ having the
type of $\ptl{p}$ is a regular Euclidean $m$-gon. Define a straight-edge
triangulation of $\atl{a}$ in the pattern of the triangulation of
$\ptl{p}$, with corners going to corresponding corners. This can be
easily done, for example, by a Tutte embedding \cite{wtT63}.

The upshot is that we have decomposed $\ptl{p}$ and $\atl{a}$ into
combinatorially equivalent patterns of Euclidean triangles. If
$\etl{t}\in\etl{T}$ has the type of $\ptl{p}$, we may transfer the
triangulation of $\ptl{p}$ by similarity to $\etl{t}$, respecting
corner designations.  Now, define a continuous map
$g_{\etl{t}}:\etl{t}\rightarrow \atl{a}$ by mapping each of these
triangles of $\etl{t}$ affinely onto the corresponding triangle of
$\atl{a}$. Affine maps are quasiconformal and we may take $\kappa$ to
be the maximum of dilatations not only for the finitely many triangles of
$\ptl{p}$, but for the finite number of triangles over all prototile
types. Thus $g_{\etl{t}}$ is $\kappa$-quasiconformal on the interior
of $\etl{t}$.

Applying the construction to every tile of $\etl{T}$ and noting
that when tiles $\etl{t},\etl{t}'$ share an edge $e$, the maps
$g_{\etl{t}},g_{\etl{t}'}$, being affine, agree on $e$, we obtain
a continuous map $g:\etl{T}\rightarrow \atl{A}$.
Since $g$ is $\kappa$-quasiconformal on each tile and the union 
of boundaries of the triangles has area zero, $g$ is
$\kappa$-quasiconformal on all of $\etl{T}$.

Recall that we defined a conformal structure on $\atl{A}$
which is compatible with its p.w. affine structure and a conformal map
$\phi:\atl{A}\rightarrow \ctl{T}$. The map $f:\etl{T}
\rightarrow \ctl{T}$ is defined by $f=\phi\circ g$ and, because
$\phi$ is $1$-quasiconformal, $f$ is $\kappa$-quasiconformal.

Finally, as $\etl{T}$ fills $\bC$, Liouville's theorem for quasiconformal
mapping \cite{LV73} implies that the image $\ctl{T}$ fills $\bC$ as well.
So $\ctl{T}$ is parabolic.
\end{proof}

\subsection{Reduction}
Observe that for any $n\ge 0$ we may define the
$\kappa$-quasiconformal map $F_n:\etl{T}_n\rightarrow \ctl{T}_n$ in
the fashion of Lemma~\ref{L:qc} and that we have a uniform $\kappa$,
since $\kappa(\etl{T})=\kappa(\etl{T}_n)$ for all $n\ge0$.  Since each
$F_n$ fixes the points 0 and 1 by our earlier normalization, it is a
normal family: there exists a subsequence $\{F_{n_j}\}$ which
converges to a $\kappa$-quasiconformal limit function $F$ which also
fixes $0,1$: thus $F_{n_j}\rightarrow F$ uniformly on compacta in
$\bC$ as $n_j\rightarrow\infty$.  We will prove that $F$ is the
identity map. Then, since the aggregate conformal tiles
$\ctl{t}\bs{n}=F_n(\etl{t}\bs{n})$ converge pointwise to
$F(\etl{t}\bs{n})$ and since $\etl{t}\bs{n}=\etl{t}$, the fact that
$F(\etl{t})=\etl{t}$ will imply the desired conclusion,
$\ctl{t}\bs{n}\cis\etl{t}$. In hindsight we may observe that in fact
the full sequence $\{F_n\}$ converges to the identity.

To prove that $F$ is the identity, recall that as a
$\kappa$-quasiconformal map it has a derivative $dF(z)$ for almost
every (with respect to Lebesgue measure) $z\in\bC$. If $dF(z)$ exists
and is a similarity, then the dilatation of $F$ at $z$ must be 1. If
the dilatation is 1 a.e., then $F$ is $1$-quasiconformal --- that is,
$F$ is an entire function. Since $F$ is a homeomorphism fixing $0$ and
$1$, we can conclude that $F$ is the identity.

To complete our proof, therefore, it is enough to show that the linear
mapping $L_z$ associated with $dF(z)$ is a similarity for almost all
$z\in\bC$. We fix attention on a point $z_0$ where $dF(z_0)$ exists.
Translating $\etl{T}$ by $-z_0$ and the mappings $F_n$ by $-F_n(z_0)$,
we may assume without loss of generality that $z_0=0$ and
$F_n(0)=F(0)=0$.  The remainder of the proof then consists in proving
this

\begin{Claim} $L=L_0$ is a similarity.
  \end{Claim}

\subsection{Proof of the Claim}
A bit of notation first. We currently have extracted the subsequence
$\{F_{n_j}\}$ with 
\begin{equation}\tag{A}
  F_{n_j}\rightarrow F\text{ uniformly on compacta as }n_j\rightarrow \infty.
\end{equation}
A finite number of further extractions will be needed, so we abuse notation by
referring in each instance to $\{n_j\}$ as the latest subsequence. We
will also frequently use index ``$n$'' when referring to ``$n_j$'';
context should make our intentions clear. 

To zoom in at 0, we introduce {\sl blowups}. Given
$\psi:\bC\rightarrow \bC$ with $\psi(0)=0$ and given $\rho>0$, define
$\bup{\psi}{\rho}(z)=\psi(\rho z)/\rho, z\in \bC$. Note that
$\bup{\psi}{\rho}(0)=0$ and if the differential $d\psi(0)$ exists,
then $d\psi(0)=d\bup{\psi}{\rho}(0)$, and the limit as $\rho\to0$ is a
linear map:
$$\lim_{\rho\to0}\bup{\psi}{\rho}(z) = [d\bup{\psi}{\rho}(0)](z).$$

The blowups to consider are $\bup{F}{\rho}_n$ and $\bup{F}{\rho}$.
Observe that $d\bup{F}{\rho}(0)=dF(0)$ for all $\rho$, implying that
\begin{equation}\tag{B}
  \bup{F}{\rho}\longrightarrow L\text{ uniformly on compacta of }\cD
  \text{ as }\rho\downarrow 0.
\end{equation}
Moreover, due to (A), for each fixed $\rho>0$ and $n\ge 0$,
\begin{equation}\tag{C}
  \bup{F}{\rho}_{n+m}\rightarrow \bup{F}{\rho}\text{ uniformly
    on compacta as }m\rightarrow \infty.
\end{equation}
(Note an additional abuse of notation here: As $n$ denotes $n_j$, so
$n+m$ here denotes $n_{j+m}$ so that we may appeal to (A).  We will
assume this meaning without further comment when we apply (C).)

\subsubsection{In the Domain}
We may henceforth restrict attention to some disc $\cD$ centered at
0 (e.g., $\cD=\bD$) as the common domain.
The special configuration $\mC$ defined in \S\ref{S:subdetails} is a
union of two congruent copies of some prototile $\ptl{p}$.  A copy of
$\mC$ among the tiles of $\etl{T}_n$ occurs in arbitrarily small
neighborhoods of $0$ for sufficiently large $n$. Therefore, we can
extract a (further) subsequence $\{F_{n_j}\}$ and a corresponding
sequence $\rho_{n_j}$ of blowup parameters, $\rho_{n_j}\downarrow 0$,
so that the scaled tilings ${\etl{T}_{n_j}}/{\rho_{n_j}}$ contain
configurations $\mC_{n_j}\subset \cD$ similar to $\mC$ and with
diameters bounded above and below. By compactness in the Hausdorff
metric we may extract a further subsequence $\{n_j\}$ so that
$\mC_{n_j}\rightarrow \wmC$ where the limit configuration
$\wmC\subset \cD$ is again similar to $\mC$.

For each $n=n_j$, the configuration $\mC_{n}$ is a union $p_n\cup q_n$
of scaled tiles from ${\etl{T}_n}/{\rho_n}$. Here $p_n$ and $q_n$ are
congruent {\it via} a similarity
$S_n:p_n\rightarrow q_n$ that is not a translation. Likewise, in the
limit $\wmC$
is a union of polygons $p$ and $q$ that are congruent to one another
{\sl via} a similarity $S:p\rightarrow q$ that is not a
translation. We note the following convergence properties for
later use:
\begin{equation}\tag{D}
  p_n\rightarrow p,\ \ q_n\rightarrow q,\text{ and }S_n\rightarrow S,
  \text{ as } n=n_j\rightarrow \infty.
\end{equation}

\subsubsection{In the Range} In the range \hide{common range $\wmD$}
we have several additional sets to consider. Let $\tilde p,\tilde q$
be the Euclidean polygons $\tilde{p}=L(p)$ and $\tilde{q}=L(q)$ and
define the affine transformation
$\widetilde{S}:\tilde{p}\rightarrow\tilde{q}$ by $\widetilde{S}=L\circ
S\circ L^{-1}$. Note that, as usual,
$\widetilde{S}:\tilde{p}\rightarrow\tilde{q}$ identifies the
distinguished corners of $\tilde{p}$ and $\tilde{q}$.  The conclusion
will follow once we show that $\widetilde{S}$ is conformal.

Consider the images of aggregates associated with the
tiles $p_n,q_n$. For each $m\ge 0$, the tiles $p_{n}$ and $q_{n}$ in
$\cD$ may be subdivided $m$ times by $\tau$ to give the $m$-aggregate
tiles $p_{n}\bs{m}=\tau^mp_{n}$ and $q_{n}\bs{m}=\tau^mq_{n}$
within $\etl{T}_{n+m}/\rho_n$. Of course, since this subdivision
takes place {\it in situ} in the domain, $p_{n}\bs{m}=p_{n}$ and
$q_{n}\bs{m}=q_{n}$ as point sets in $\cD$.  However, the additional
$m$ subdivisions have implications for the conformal images in the range.
We denote the aggregate images by
\begin{equation*}
  \tilde{p}_{n,m}=\bup{F}{\rho_n}_{n+m}(p_n\bs{m})=\bup{F}{\rho_n}_{n+m}(p_{n}),
\end{equation*}
\begin{equation*}
  \tilde{q}_{n,m}=\bup{F}{\rho_n}_{n+m}(q_n\bs{m})=\bup{F}{\rho_n}_{n+m}(q_{n}).
\end{equation*}
The last bit of tile notation is this:
\begin{equation*}
\hat{p}_n=\bup{F}{\rho_n}(p_n)\ \ \text{ and
}\ \ \hat{q}_n=\bup{F}{\rho_n}(q_n).
  \end{equation*}
Using (C) we have:
\begin{equation}\tag{E}
  \tilde{p}_{n,m}\rightarrow \hat{p}_n,\qquad
  \tilde{q}_{n,m}\rightarrow \hat{q}_n,\ \ \text{ as }\ m\rightarrow \infty.
\end{equation}

In the range we define the homeomophisms
$\Phi_{n,m}:\tilde{p}_{n,m}\rightarrow \tilde{q}_{n,m}$ by
\begin{equation*}
  \Phi_{n,m}=\bup{F}{\rho_n}_{n+m}\circ S_n\circ (\bup{F}{\rho_n}_{n+m})^{-1}.
\end{equation*}
We claim that $\Phi_{n,m}$ is in fact conformal. Recall that
the quasiconformal maps $F_{n+m}$ were defined in a tile-by-tile
fashion. The action on each tile $\etl{t}$ having the type of, say,
$\ptl{p}_j$, was modeled on a composition $f_j\circ g_j$, where $g_j$
is a piecewise affine map carrying some fixed triangulation of $\ptl{p}_j$
to an equivalent triangulation of a regular Euclidean $n$-gon and $f_j$
is a conformal map of that $n$-gon to the conformal image 
tile. But each subtile of $\tilde{p}_{n,m}$ is mapped by $\Phi_{n,m}$ to
a subtile of $\tilde{q}_{n,m}$ sharing the same tile type, so the composition
defining $\Phi_{n,m}$ is modeled on compositions of the form
\begin{equation*}f_j\circ g_j\circ S\circ g_j^{-1} \circ f_j^{-1}.
\end{equation*}
Since $S$ is a similarity and the $g_j$ is canonical for the given tile
type, the center three factors give a similarity. Since $f_j$ is conformal,
this means that the restriction of $\Phi_{n,m}$ to each subtile of
$\tilde{p}_{n,m}$ is conformal. The edges between subtiles form a set of
area zero, so $\Phi_{n,m}:\tilde{p}_{n,m}\rightarrow \tilde{q}_{n,m}$ is
a conformal mapping.

Noting the convergence in (E) and applying normal families and
the Carathe\'odory Kernel Theorem from conformal function theory,
we obtain a limit conformal mapping
$\Phi_n:\hat{p}_n\rightarrow \hat{q}_n$. More explicitly, recalling the
definition of $\Phi_{n,m}$ and the fact that $p_{n,m}=p_n$ and $q_{n,m}=q_n$,
we have \begin{equation*}
\Phi_n=\bup{F}{\rho_n}\circ S_n\circ (\bup{F}{\rho_n})^{-1},\qquad
  \Phi_n:\hat{p}_n \xrightarrow{(\bup{F}{\rho_n})^{-1}} 
  p_n \xrightarrow{S_n} 
  q_n \xrightarrow{\bup{F}{\rho_n}}\hat{q}_n.
\end{equation*}

\subsubsection{Conclusion}
To conclude the proof of the claim, note that as we let $n=n_j$ go to
infinity, we have $\rho_n$ going to zero as well. Using (B) and (D) we
see that $\hat{p}_n\rightarrow \tilde{p}$ and $\hat{p}_n\rightarrow
\tilde{p}$. The sequence $\{\Phi_n\}$ is a normal family, and so up to
taking a subsequence the $\Phi_n$ converge to a limit conformal
mapping $\Phi$,
\begin{equation*}
  \Phi_n\rightarrow \Phi:\tilde{p}\rightarrow \tilde{q}
\end{equation*}
Since by (B) the blowups $\bup{F}{\rho_n}$ converge uniformly on
compacta of $\cD$ to the linear transformation $L$, we see that
$\Phi=L\circ S\circ L^{-1}$ on $\tilde{p}$. In particular,
$L\circ S\circ L^{-1}$ is a similarity. Since $S$ is a similarity
whose linear part is not $\pm I\,$, then the fact that
$L\circ S\circ L^{-1}$ is a similarity implies $L$ must be a
similarity. This completes the proof of the Claim and hence of the
Theorem.

\section{Examples and Questions}\label{S:Examples}
\F{AggShapes} illustrates four substitution tilings from the
traditional tiling literature: the ``chair'', ``domino'', ``sphinx'',
and ``Penrose''. The Euclidean shapes are shown with their subdivision
schemes. Note that the Penrose tile here is a ``dart'' from the familiar
``kite/dart'' version of Penrose tilings; in actuality, there are four 
Robinson prototiles involved, and all four appear in the kite subdivsion.

\vspace{20pt}
\begin{center}
\begin{overpic}[width=.9\textwidth
]{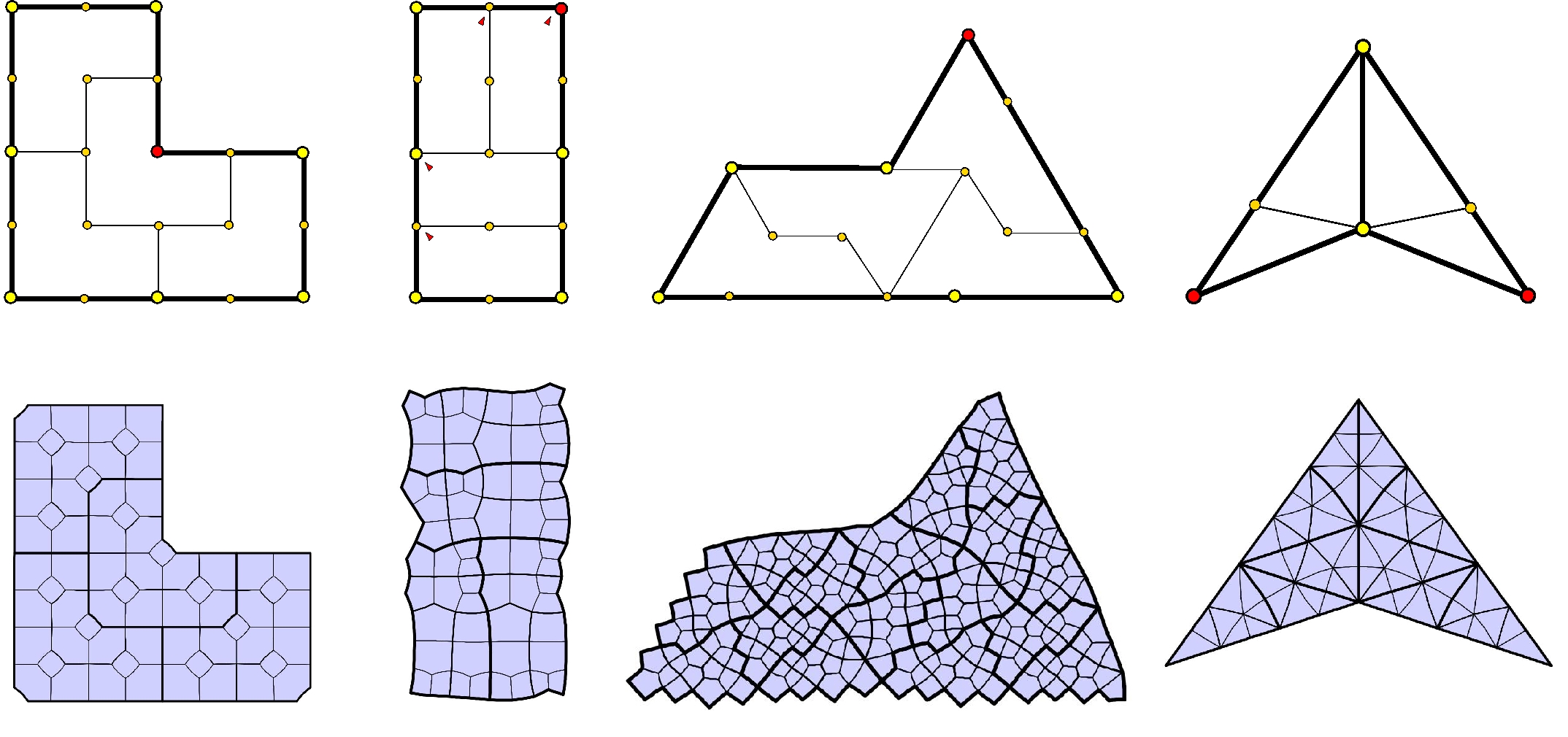}
\end{overpic}
\captionof{figure}{Along the top: chair, domino, sphinx, and Penrose
  Euclidean tiles and their subdivisions. Along the bottom, associated
  aggregate tiles isolated from conformal tilings. The two on the left
  are 3-aggregates, the remaining two,
  4-aggregates.}\label{F:AggShapes}
\end{center}

An intriguing aspect of shape convergence is that the purely
combinatorial data in $\ktl{K}$ and $\tau$ encode precise Euclidean
shapes. Thus the combinatorics of the pinwheel knows about $\sqrt{5}$,
those of the domino know the aspect ratio 1:2, the Penrose kite knows
the golden ratio hidden throughout its tilings. This connection
between combinatorics and geometry has many precedents, of course. The
most farreaching, perhaps, is in Gorthendieck's {\sl dessins
  d'Enfants}, wherein abstract finite graphs (``child drawings'') lead
to algebraic number fields (see \cite{aG85-t,BS04}). Another deep
connection is proposed in Cannon's Conjecture, \cite{jC94,CFP1,CFP2},
concerned with the recognition of Klienian groups from the
combinatorics of subdivision operators.  In conformal tiling itself,
the foundational example was the pentagonal tiling studied in
\cite{BS97}. Jim Cannon, Bill Floyd, and Walter Parry, along with the
first author, Rick Kenyon, proved in \cite{CFKP97} that it arises from
iteration of the inverse of a rational function with integer
coefficients, leading, for example, to the wonderful scaling factor of
the tiling, $\lambda=(324)^{-1/5}$.

In this broader view, the combinatorics of substitution tilings are
rather rarified --- examples are difficult to come by and the handful
available are prized. In contrast, the combinatorics of conformal
tilings can be nearly arbitrary, even if one specifies finite numbers
of tile types and finite local complexity. Several examples in
\cite{BS17} not associated with traditional substitution tilings
illustrate this ubiquity. Typically the finite number of combinatorial
tile types have, in their conformal tilings, infinitely many Euclidean
shapes. Although general limiting behaviours of aggregate tiles have
not yet been studied closely in the conformal setting, \cite{BS17}
does identify an important class of ``conformal'' subdivision rules
$\tau$, under which tile shapes, though infinite in variety, still
subdivide {\sl in situ}.

This landscape of combinatorial/geometric interactions raises some natural 
questions which we pose to the interested reader. 

\begin{Question} Are there criteria to determine whether
  a given finite collection of combinatorial tile types with a
  combinatorial subdivision operator $\tau$ is associated with a
  substitution tiling?
\end{Question}

\begin{Question} Can one discover new substitution tilings in this way?
\end{Question} 

\begin{Question} What is the limiting fate of aggregate conformal tiles for 
general combinatorial subdivision tilings?
\end{Question}

A last comment is about the experiments behind this paper: Circle
packing is as yet the only method for approximating conformal tilings
in practice.  The examples illustrated here were created in the
software package \CP.  This software is available on the third
author's web site. In addition, the \CP\ scripts for creating and
manipulating the specific examples in the paper are available from the
third author on request.

\section{Acknowledgements}
Research of R.K. is supported by NSF grant  DMS-1612668 and the Simons Foundation award 327929; research of K.S. is supported by Simons Foundation award 208523.
We thank Rich Schwartz for valuable conversations.


\bibliographystyle{amsplain}
\bibliography{OurBib} 
\end{document}